\documentclass[11pt]{article}
\usepackage{amsthm,amssymb,amsmath}
\usepackage{epsfig}

\title{Crossing-critical graphs with large maximum degree}

\author{Zden\v{e}k Dvo\v{r}\'ak\thanks{Supported in part through a postdoctoral
   position at Simon Fraser University.}~\thanks{On leave from: Institute of Theoretical Informatics,
   Charles University, Prague, Czech Republic.}\\
  {Department of Mathematics}\\
  {Simon Fraser University}\\
  {Burnaby, B.C. V5A 1S6} \\
  email: {\tt rakdver@kam.mff.cuni.cz}
\and
  Bojan Mohar\thanks{Supported in part by the
  Research Grant P1--0297 of ARRS (Slovenia), by an NSERC Discovery Grant (Canada)
  and by the Canada Research Chair program.}~\thanks{On leave from:
  IMFM \& FMF, Department of Mathematics, University of Ljubljana, Ljubljana,
  Slovenia.}\\
  {Department of Mathematics}\\
  {Simon Fraser University}\\
  {Burnaby, B.C. V5A 1S6} \\
  email: {\tt mohar@sfu.ca}
}

\newtheorem{theorem}{Theorem}
\newtheorem{lemma}[theorem]{Lemma}

\newtheorem{conjecture}[theorem]{Conjecture}

\def\crn{\text{\rm cr}}
\def\crit{\text{\rm crit}}
\newcommand{\DEF}[1]{{\em #1\/}}

\begin{document}

\maketitle

\begin{abstract}
A conjecture of Richter and Salazar about graphs that are critical for a fixed 
crossing number $k$ is that they have bounded bandwidth. 
A weaker well-known conjecture of Richter is that 
their maximum degree is bounded in terms of $k$. In this note we disprove these
conjectures for every $k\ge 171$, by providing examples of $k$-crossing-critical 
graphs with arbitrarily large maximum degree.
\end{abstract}


A graph is \DEF{$k$-crossing-critical} (or simply \DEF{$k$-critical})
if its crossing number is at least $k$,
but every proper subgraph has crossing number smaller than $k$. 
Using the Excluded Grid Theorem of Robertson and Seymour \cite{RSey}, it is not 
hard to argue that $k$-crossing-critical graphs have bounded tree-width~\cite{GRS}. 
However, all known constructions of crossing-critical graphs 
suggested that their structure is ``path-like''. 
Salazar and Thomas conjectured (cf.~\cite{GRS}) that they have bounded path-width.
This problem was solved by Hlin\v{e}n\'y \cite{Hl}, who proved that the path-width 
of $k$-critical graphs is 
bounded above by $2^{f(k)}$, where $f(k) = (432\log_2 k + 1488) k^3 + 1$.

In the late 1990's, two other conjectures were proposed and made public
in 2003 at the Bled'03 conference \cite{Bled}
(see also \cite{RS} and \cite{MPRTT}).

\begin{conjecture}[Richter \cite{Bled}]
\label{conj:1}
For every positive integer $k$, there exists an integer $D(k)$ 
such that every $k$-crossing-critical graph has maximum degree 
less than $D(k)$.
\end{conjecture}

The second conjecture was proposed as an open problem in the 1990's by
Carsten Thomassen and formulated as a conjecture by Richter and Salazar.

\begin{conjecture}[Richter and Salazar \cite{Bled,RS}]
\label{conj:2}
For every positive integer $k$, there exists an integer $B(k)$ 
such that every $k$-crossing-critical graph has bandwidth at most $B(k)$.
\end{conjecture}

Conjecture \ref{conj:2} would be a strengthening of Hlin\v{e}n\'y's theorem
about bounded path-width and would also imply Conjecture \ref{conj:1}.

Hlin\v{e}n\'y and Salazar \cite{HS} recently made a step towards 
Conjecture \ref{conj:1} by proving that $k$-crossing-critical graphs cannot
contain a subdivision of $K_{2,N}$ with $N=30k^2+200k$.

In this note we give examples of $k$-crossing-critical graphs of arbitrarily large 
maximum degree, thus disproving both Conjectures \ref{conj:1} and~\ref{conj:2}.

\medskip

A \DEF{special graph} is a pair $(G,T)$, where $G$ is a graph and $T\subseteq E(G)$.
The edges in the set $T$ are called \DEF{thick edges} of the special graph.
A \DEF{drawing} of a special graph $(G,T)$ is a drawing of $G$ such that the edges 
in $T$ are not crossed. The crossing number $\crn(G,T)$ of a special graph is 
the minimum number of edge crossings in a drawing of $(G,T)$ in the plane.
(We set $\crn(G,T)=\infty$ if a thick edge is crossed in every drawing of $G$.)
An edge $e\in E(G)\setminus T$ is
\DEF{$k$-critical} if $\crn(G,T)\ge k$ and $\crn(G-e,T)<k$.  Let $\crit_k(G,T)$ be
the set of $k$-critical edges of $(G,T)$.  
If $T=\emptyset$, then we write just $\crn(G)$
for the crossing number of $G$ and $\crit_k(G)$ for the set of $k$-critical
edges of $G$. Note that the graph $G$ is $k$-critical if $\crit_k(G)=E(G)$.

A standard result (see, e.g., \cite{DMS}) is that we can eliminate the
thick edges by replacing them with sufficiently dense subgraphs. 
(In fact, one can replace every edge $xy$ by $t=\crn(G,T)+1$ parallel edges
or by $K_{2,t}$ if multiple edges are not desired.) 

\begin{lemma}\label{lemma-elimthick}
For every special graph $(G,T)$ with $\crn(G,T)<\infty$
and for any $k$, there exists a graph 
$\tilde G\supseteq G$ such that $\crn(G,T)=\crn(\tilde G)$ and\/ $\crit_k(G,T)\subseteq \crit_k(\tilde G)$.
\end{lemma}

Furthermore, note the following:

\begin{lemma}\label{lemma-elimextra}
Let $k$ be an integer. Any graph $G$ with $\crn(G)\ge k$ contains 
a $k$-crossing-critical subgraph $H$ such that $\crit_k(G)\subseteq E(H)$.
\end{lemma}

\begin{proof}
For a contradiction, suppose that $G$ is a smallest counterexample.
If $G$ were $k$-critical, then we would set $H=G$, hence $G$ contains
a non-$k$-critical edge $e$.  It follows that $\crn(G-e)\ge k$.
Let $f$ be a $k$-critical edge in $G$, i.e., $\crn(G-f)<k$.
As $\crn((G-e)-f)\le\crn(G-f)<k$, $f$ is a $k$-critical edge in $G-e$.
Therefore, $\crit_k(G)\subseteq\crit_k(G-e)$.  Since $G$ is the smallest
counterexample, $G-e$ has a $k$-critical subgraph $H$ with
$\crit_k(G-e)\subseteq E(H)$.  However, $H\subseteq G$ and $\crit_k(G)\subseteq E(H)$,
which is a contradiction.
\end{proof}

Let us now proceed with the main result. Two paths $P_1$ and $P_2$ in
a special graph are \DEF{almost edge-disjoint} if all the edges in 
$E(P_1)\cap E(P_2)$ are thick.

\begin{figure}
\centering
\includegraphics[width=1.0\textwidth]{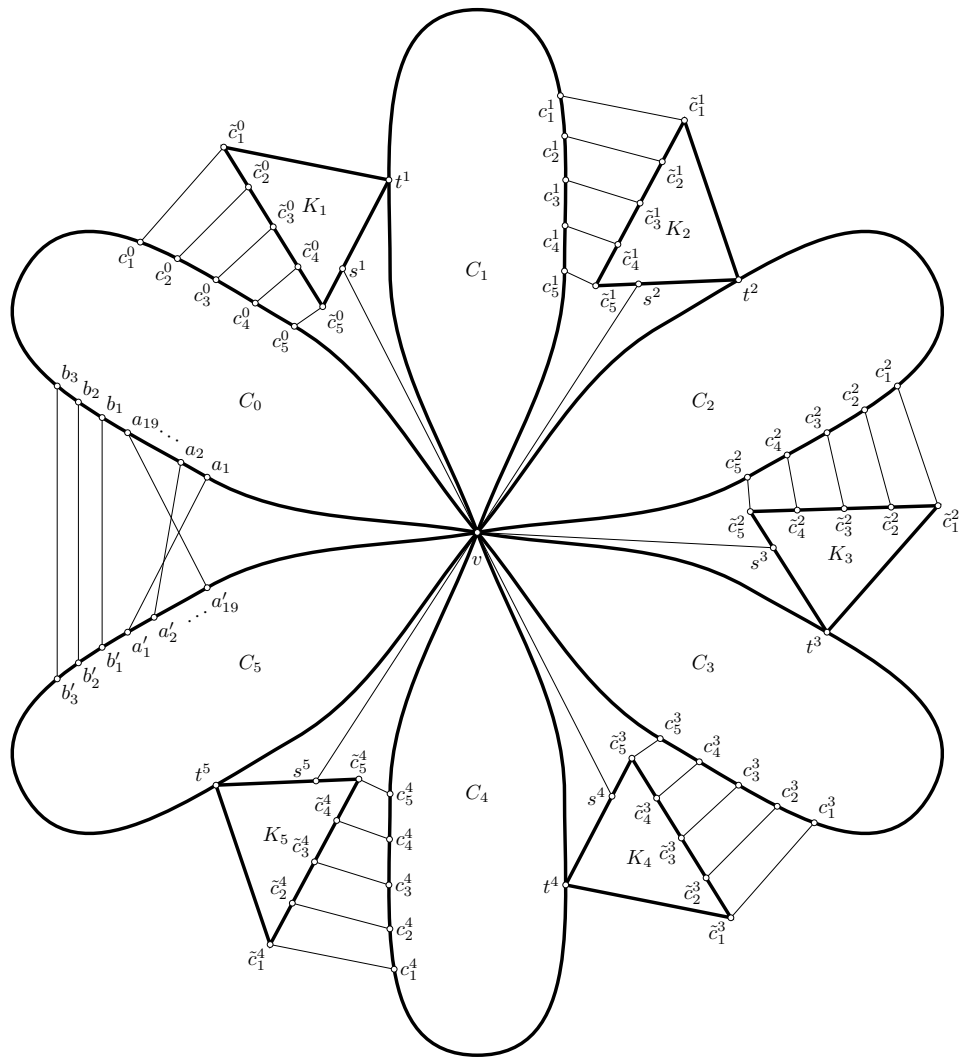}
\caption{A special graph with critical edges $vs^i$}
 \label{fig:1}
\end{figure}

\begin{lemma}\label{lemma-splarge}
For any $d$, there exists a special graph $(G,T)$ and a vertex
$v\in V(G)$ such that $\crit_{171}(G,T)$
contains at least $d$ edges incident with $v$.
\end{lemma}

\begin{proof}
Let $(G,T)$ be the special graph drawn as follows:  we start with $d+1$
thick cycles $C_0,C_1,\ldots,C_d$ intersecting in a vertex $v$,
i.e., $C_i\cap C_j = \{v\}$ for $0\le i < j \le d$. Their lengths are
$|C_0|=28$, $|C_d|=24$ and $|C_i|=7$ for $1\le i<d$.
They are drawn in the plane so that all their vertices are incident with the unbounded face and
their clockwise order around $v$ is $C_0,C_1,\ldots,C_d$.
See Figure \ref{fig:1} illustrating the case $d=5$.
Let $C_0=va_1a_2\ldots a_{19}b_1b_2b_3c^0_1c^0_2\ldots c^0_5$, $C_d=vt^db'_3b'_2b'_1a'_1a'_2\ldots a'_{19}$
and $C_i=vt^ic_1^ic_2^i\ldots c_5^i$ for $1\le i<d$.  Furthermore, 
add $d$ vertices $s^1$, \ldots, $s^d$ adjacent to $v$.
The clockwise cyclic order of the neighbors of $v$ is 
$a_1,c^0_5,s^1,t^1,c^1_5,s^2,t^2,c^2_5, \ldots, 
s^{d-1},t^{d-1},c^{d-1}_5, s^d,t^d, a'_{19}$.
For $1\le i\le d$, add thick cycles $K_i$ whose vertices in the clockwise order are $t^i$, $s^i$,
and five new vertices $\tilde{c}^{i-1}_5$, $\tilde{c}^{i-1}_4$, \ldots, $\tilde{c}^{i-1}_1$.  Finally, add the following
edges: $c^i_j\tilde{c}^i_j$ for $0\le i<d$ and $1\le j\le 5$, $a_ia'_i$ for $1\le i\le 19$
and $b_ib'_i$ for $1\le i\le 3$.  As described, $T=\bigcup_{i=0}^d E(C_i) \cup \bigcup_{i=1}^d E(K_i)$.
Let $M=\{a_1a'_1,a_2a'_2,\ldots,a_{19}a'_{19},b_1b'_1,b_2b'_2,b_3b'_3\}$.

This drawing ${\cal G}$ of $(G,T)$ has ${19\choose2}=171$ crossings, as the edges 
$a_ia'_i$ and $a_ja'_j$ intersect for each $1\le i<j\le 19$, and
there are no other crossings. Let us show that $\crn(G,T)=171$.  
Let ${\cal G}'$ be an arbitrary drawing of
$(G,T)$, and for a contradiction assume that it has less than $171$ crossings.
Let us first observe that every thick cycle $C_i$ and $K_j$ is an induced 
nonseparating cycle of $G$. Therefore it bounds a face of ${\cal G}'$.
Consider the cyclic clockwise order of the neighbors of $v$ according to the
drawing ${\cal G}'$. For each cycle $C_i$ ($0\le i\le d$), 
the two edges of $C_i$ incident with $v$ are consecutive in this order, 
since $C_i$ bounds a face. Without loss on generality, we assume that each 
cycle $C_i$ bounds a face distinct from the unbounded one.
If the cyclic order of the vertices around the face $C_i$ is
the same as in the drawing ${\cal G}$, we say
that $C_i$ is drawn {\em clockwise}, otherwise it is drawn {\em anti-clockwise}.
We may assume that $C_0$ is drawn clockwise.  If $C_d$ were drawn clockwise 
as well, then each pair of
edges $a_ia'_i$ and $a_ja'_j$ with $1\le i<j\le 19$ would intersect,
and the drawing ${\cal G}'$ would have
at least $171$ crossings. Therefore, $C_d$ is drawn anti-clockwise.
It follows that the edges $a_ia_i'$ and $b_jb_j'$ intersect for $1\le i\le 19$ 
and $1\le j\le 3$, and the edges $b_ib'_i$ and $b_jb'_j$ intersect for 
$1\le i<j\le 3$, giving $60$ crossings. For $1\le i\le 5$, let $P_i$ be the path
$c^0_i\tilde{c}^0_i\tilde{c}^0_{i-1}\ldots\tilde{c}^0_1t^1c^1_1c^1_2\ldots c^1_i\tilde{c}^1_i\ldots 
\tilde{c}^1_1t^2 \ldots t^d$.
These paths are mutually almost edge-disjoint and each of them intersects all edges
of $M$ in the drawing ${\cal G}'$, thus contributing at least $110$ crossings
all together. Therefore, the drawing ${\cal G}'$ has at least $170$ crossings.  
Since we assume that this drawing
has less than $171$ crossings, we conclude that there are no other crossings.

The cycle $va_1a'_1a'_2\ldots a'_{19}$ splits the plane into two regions $R_1$ and
$R_2$, such that $R_1$ contains the face bounded by $C_0$ and $R_2$
contains the face bounded by $C_d$.
For $j=1,2$, let $A_j$ be the set of cycles $C_i$ ($0\le i\le d$) such that
the face bounded by $C_i$ lies in the region $R_j$.
As $P_1$ intersects the edge $a_1a'_1$ only once, $A_1=\{C_0,C_1,\ldots, C_{k-1}\}$
and $A_2=\{C_k,C_{k+1},\ldots, C_d\}$ for some $k$ with $1\le k\le d$.
As the path $P_1$ does not intersect itself, all cycles in $A_1$
are drawn clockwise and their clockwise order around $v$ is $C_0$, $C_1$,
\ldots, $C_{k-1}$.  Similarly, all cycles in $A_2$ are drawn anti-clockwise
and their clockwise order around $v$ is $C_d$, $C_{d-1}$, \ldots, $C_k$.

Let us now consider the cycle $K_k$.  Since the edges 
$c^{k-1}_4\tilde{c}^{k-1}_4$ and $c^{k-1}_5\tilde{c}^{k-1}_5$
do not intersect, the thick path $c^{k-1}_5vt^ks^k\tilde{c}^{k-1}_5$ is not intersected,
and $C_{k-1}$ is drawn clockwise, $K_k$ is drawn clockwise as well.
Since $C_k$ lies in the region $R_2$, the vertex $t^k$ and thus the whole thick cycle
$K_k$ lie in $R_2$.  However, that means that the edge $s^kv$
intersects either the path $P_1$ or the edge $a_1a'_1$, which is a contradiction.
We conclude that $\crn(G,T)=171$.

On the other hand, $\crn(G-vs^k,T)<171$, for $1\le k\le d$ 
(in fact, $\crn(G-vs^k,T)=170$). To see that, consider the drawing 
of $(G-vs^k,T)$ in which the cycles $C_0$, $C_1$, \ldots, $C_{k-1}$
are drawn clockwise, the cycles $C_k$, $C_{k+1}$, \ldots, $C_d$ are drawn 
anti-clockwise, and the cyclic order of the neighbors of $v$ is
$a_1c^0_5s^1t^1c^1_5\ldots s^{k-1}t^{k-1}c^{k-1}_5a'_{19} t^dc^{d-1}_5s^{d-1}t^{d-1}\ldots c^k_5t^k$.
The intersections of this drawing are of edges $a_ia_i'$ with $b_jb_j'$ for $1\le i\le 19$ and $1\le j\le 3$,
the edges $b_ib'_i$ with $b_jb'_j$ for $1\le i<j\le 3$,
and the edges $c_i^{k-1}\tilde{c}_i^{k-1}$ with all edges of $M$ for $1\le i\le 5$.
Therefore, the edge $vs^k$ is $171$-critical for each $k$, so $v$ is incident 
with $d$ critical edges.
\end{proof}

We are ready for our main result.

\begin{theorem}
\label{thm:main}
For every $k\ge 171$ and every $d$, there exists a $k$-crossing-critical graph $H$
containing a vertex of degree at least $d$.
\end{theorem}

\begin{proof}
Let $(G,T)$ be the special graph constructed in Lemma~\ref{lemma-splarge}.
By Lemma~\ref{lemma-elimthick}, there exists a graph $H'\supseteq G$ such that
$\crn(H')=\crn(G,T)\ge 171$ and $\crit_{171}(G,T)\subseteq \crit_{171}(H')$.
Let $H$ be the $171$-critical subgraph of $H'$ obtained by Lemma~\ref{lemma-elimextra}.
As $\crit_{171}(G,T)\subseteq\crit_{171}(H')\subseteq E(H)$, $H$ contains at least $d$
edges incident with one vertex, hence $\Delta(H)\ge d$.
For $k > 171$ we add to $H$ $k-171$ copies of the graph $K_5$ in order to get
a $k$-crossing-critical graph.
\end{proof}

Actually, in the proof of Theorem \ref{thm:main}, we can take 
$t=\left\lfloor \tfrac{k}{171}\right\rfloor$ copies of the graph $H$ and 
$k-171t$ copies of $K_5$. This gives rise to a $k$-critical graph with
$t=\Omega(k)$ vertices of (arbitrarily) large degree. We conjecture that this
is best possible in the following sense:

\begin{conjecture}
For every positive integer $k$ there exists an integer $D=D(k)$ such that
every $k$-crossing-critical graph contains at most $k$ vertices whose degree 
is larger than $D$.
\end{conjecture}

It is not even obvious if there exist $k$-crossing-critical graphs with arbitrarily
many vertices of degree more than 6. Surprisingly, such examples have been constructed 
recently by Hlin\v{e}n\'y \cite{Hl2}. His examples may contain arbitrarily many
vertices of any even degree smaller than $2k-1$.

\end{document}